\documentclass[12pt, reqno]{amsart}

\usepackage{hyperref}
\usepackage{amssymb, amsthm, amsmath, amsfonts,epsfig,bbm}
\usepackage[numbers,square]{natbib}
\usepackage{color}
\usepackage{hhline}
\usepackage[left=2.3cm, top=2.3cm,bottom=2.3cm,right=2.3cm]{geometry}

\usepackage[numbers,square]{natbib}
\usepackage{array}
\usepackage{makecell,tabularx}
\usepackage[british]{babel}
\usepackage{amsfonts}              

\numberwithin{equation}{section}

\newtheorem {theorem}{Theorem}[section]
\newtheorem {proposition}[theorem]{Proposition}
\newtheorem {lemma}[theorem]{Lemma}

{\theoremstyle{definition}

\newtheorem {example}[theorem]{Example}
\newtheorem {remark}[theorem]{Remark}
}

{
\theoremstyle{remark}
}

\def\E{\mathbb{E}}

\def\N{\mathbb{N}}
\def\P{\mathbb{P}}

\def\R{\mathbb{R}}
\def\RRd1{\mathbb{R}^{d+1}}

\def\G{\Gamma}

\def\L{\Lambda}

\def\cA{\mathcal{A}}
\def\cB{\mathcal{B}}

\def\cF{\mathcal{F}}

\def\cL{\mathcal{L}}

\def\cR{\mathcal{R}}

\def\cW{\mathcal{W}}

\newcommand{\ind}{\mathbbm{1}}
\newcommand{\eps}{\varepsilon}
\newcommand{\pos}{\mathop{\mathrm{pos}}\nolimits}
\newcommand{\aff}{\mathop{\mathrm{aff}}\nolimits}
\newcommand{\lin}{\mathop{\mathrm{lin}}\nolimits}
\newcommand{\rank}{\mathop{\mathrm{rank}}\nolimits}

\DeclareMathOperator{\relint}{relint}
\DeclareMathOperator{\Int}{int}

\makeatletter
\def\moverlay{\mathpalette\mov@rlay}
\def\mov@rlay#1#2{\leavevmode\vtop{%
   \baselineskip\z@skip \lineskiplimit-\maxdimen
   \ialign{\hfil$\m@th#1##$\hfil\cr#2\crcr}}}
\newcommand{\charfusion}[3][\mathord]{
    #1{\ifx#1\mathop\vphantom{#2}\fi
        \mathpalette\mov@rlay{#2\cr#3}
      }
    \ifx#1\mathop\expandafter\displaylimits\fi}
\makeatother

\newcommand{\bigcupdot}{\charfusion[\mathop]{\bigcup}{\cdot}}

\begin{document}

\title[Coefficients of characteristic polynomials of hyperplane arrangements]{An identity for the coefficients of characteristic polynomials of hyperplane arrangements}

\author{Zakhar Kabluchko}
\address{Zakhar Kabluchko: Institut f\"ur Mathematische Stochastik,
Westf\"alische Wilhelms-Universit\"at M\"unster,
Orl\'eans-Ring 10,
48149 M\"unster, Germany}
\email{zakhar.kabluchko@uni-muenster.de}

\date{}

\begin{abstract}
Consider a finite collection of affine hyperplanes in $\mathbb R^d$. The hyperplanes dissect $\mathbb R^d$ into finitely many polyhedral chambers. For a point $x\in \mathbb R^d$ and a chamber $P$ the metric projection of $x$ onto $P$ is the unique point $y\in P$ minimizing the Euclidean distance to $x$. The metric  projection is contained in the relative interior of a uniquely defined face of $P$ whose dimension is denoted by $\text{dim}(x,P)$.  We prove that for every given $k\in \{0,\ldots, d\}$,  the number of chambers $P$ for which  $\text{dim}(x,P) = k$ does not depend on the choice of $x$, with an exception of some Lebesgue null set. Moreover, this number is equal to the absolute value of the $k$-th coefficient  of the characteristic polynomial of the hyperplane arrangement. In a special case of reflection arrangements, this proves a conjecture of Drton and Klivans [A geometric interpretation of the characteristic polynomial of reflection arrangements,  \textit{Proc.\ Amer.\ Math.\ Soc.}, 138(8): 2873--2887, 2010].
\end{abstract}

\subjclass[2010]{Primary: 52C35, 51M20.	 Secondary:  52A55, 51M04, 52A22, 60D05, 52B11, 51F15}
\keywords{Hyperplane arrangement, metric projection, chambers, reflection arrangement, characteristic polynomial, normal cone, conic intrinsic volume}

\maketitle

\section{Introduction and statement of results}\label{sec:results}

\subsection{Introduction}
The starting point of the present paper was the following conjecture of Drton and Klivans~\cite[Conjecture~6]{drton_klivans}.   Consider a finite reflection group $\cW$ acting on $\R^d$. The mirror hyperplanes of the reflecting elements of $\cW$ dissect $\R^d$ into isometric cones or chambers. Let $C$ be one of these cones. Take some $k\in \{0,\ldots,d\}$. A point $x\in \R^d$ is said to have a $k$-dimensional projection onto $C$ if the unique element $y\in C$ minimizing the Euclidean distance to $x$ is contained in a $k$-dimensional face of $C$ but not in a face of smaller dimension. For example, points in the interior of $C$ have a $d$-dimensional projection onto $C$.
Drton and Klivans~\cite[Conjecture~6]{drton_klivans} conjectured that for a ``generic'' point $x\in\R^d$ the number of points in the orbit $\{gx: g\in \cW\}$  having a $k$-dimensional projection onto $C$ is  constant, that is independent of $x$. Moreover, they conjectured that this number is equal to the absolute value $a_k$ of the coefficient of $t^k$ in the characteristic polynomial of the reflection arrangement. Drton and Klivans~\cite{drton_klivans} observed that in the case of reflection groups of type $A$ their conjecture follows from the work of Miles~\cite{miles_amalgamation}, proved it for reflection groups of types  $B$ and $D$, and gave further partial results on the conjecture including numerical evidence for its validity in the case of  exceptional reflection groups. Somewhat later, Klivans and Swartz~\cite{klivans_swartz} proved that if $x$ is chosen at random according to a rotationally invariant distribution on $\R^d$, then the conjecture of Drton and Klivans is true \emph{on average}, that is the \emph{expected} number of points in the orbit $\{gx: g\in \cW\}$ having a $k$-dimensional projection onto $C$ equals $a_k$.

The aim of the present paper is to prove the conjecture of Drton and Klivans in a much more general setting of arbitrary affine hyperplane arrangements. After collecting the necessary definitions in Section~\ref{subsec:def} we shall state our main result in Section~\ref{subsec:result}.

\subsection{Definitions}\label{subsec:def}
A \emph{polyhedral set} in $\R^d$ is an intersection of finitely many closed half-spaces. A bounded polyhedral set is called a \emph{polytope}.
If the hyperplanes bounding the half-spaces pass through the origin, the intersection of these half-spaces is called a \emph{polyhedral cone}, or just a cone.   We denote by $\cF_k(P)$ the set of all $k$-dimensional faces of a polyhedral set $P\subset \R^d$, for all $k\in \{0,\ldots, d\}$. For example, $\cF_0(P)$ is the set of vertices of $P$, while $\cF_d(P) = \{P\}$ provided $P$ has non-empty interior. The set of all faces of $P$ of whatever dimension is then denoted by $\cF(P) = \cup_{k=0}^d \cF_k(P)$.  The \textit{relative interior} of a face $F$, denoted by $\relint F$,  consists of all points belonging to $F$ but not to a face of strictly smaller dimension. It is known that any polyhedral set is a disjoint union of the relative interiors of its faces:
\begin{equation}\label{eq:relint_union}
P=\bigcupdot_{F\in \cF(P)} \relint F.
\end{equation}
For more information on polyhedral sets and their faces we refer to~\cite[Sec.~7.2, 7.3]{padberg_book}, \cite{rockafellar_book} and~\cite[Chap.~1 and~2]{ziegler_book_lec_on_poly}. Polyhedral sets form a subclass of the family of closed convex sets; for the face structure in this more general setting we refer to~\cite[\S 2.1 and \S 2.4]{SchneiderBook}.

Given a polyhedral set $P$ and a point $x\in \R^d$, there is a uniquely defined point minimizing the Euclidean distance $\|x-y\|$ among all $y\in P$. This point, denoted by $\pi_P(x)$, is called the \emph{metric projection} of $x$ onto $P$. For example, if $x\in P$, then $\pi_P(x) = x$. By~\eqref{eq:relint_union},  the metric projection $\pi_P(x)$ is contained in a relative interior of a uniquely defined face $F$ of $P$.
If the dimension of $F$ is $k$, we say that the point $x$ has a \emph{$k$-dimensional metric projection} onto $P$ and write $\dim (x,P) = k$.


Next we need to recall some basic facts about hyperplane arrangements referring to~\cite{stanley_book} and~\cite[Sec.~1.7]{bona_handbook} for more information.
Let $\cA = \{H_1,\ldots,H_m\}$ be an \emph{affine hyperplane arrangement} in $\R^d$, that is a collection of pairwise distinct affine hyperplanes $H_1,\ldots, H_m$ in $\R^d$. In general, the hyperplanes are not required to pass through the origin, but if they all do, the arrangement is called  \emph{linear}. The connected components of the complement $\R^d\backslash \cup_{i=1}^m H_i$ are called \emph{open chambers}, while their closures are called \emph{closed chambers} of $\cA$. The closed chambers are polyhedral sets which cover $\R^d$ and have disjoint interiors. The collection of all \emph{closed}\footnote{This convention deviates from the standard notation~\cite{stanley_book}, where $\cR(\cA)$ is the collection of \emph{open} chambers, but will be convenient for the purposes of the present paper.} chambers will be denoted  by $\cR(\cA)$. If not otherwise stated, the word ``chamber'' always refers to a closed chamber in the sequel.
The \emph{characteristic polynomial} of the affine hyperplane arrangement $\cA$ may be defined by the following Whitney formula~\cite[Thm.~2.4]{stanley_book}:
\begin{equation}\label{eq:char_poly_def}
\chi_\cA (t) =  \sum_{\substack{\cB\subset \cA:\\\mathop{\cap}\limits_{H\in \cB} H\neq \varnothing} } (-1)^{\# \cB} t^{\dim \left(\mathop{\cap}\limits_{H\in \cB} H\right)}.
\end{equation}
Here, $\# \cB$ denotes the number of elements in $\cB$. The empty set $\cB=\varnothing$, for which the corresponding intersection of hyperplanes is defined to be $\R^d$, contributes the term $t^d$ to the above sum.  The classical Zaslavsky formulae~\cite[Thm.~2.5]{stanley_book} state that the total number of chambers is given by $\#\cR(\cA) = (-1)^d \chi_\cA(-1)$, while the number of bounded chambers is equal to $(-1)^{\rank \cA}\chi_\cA(1)$, where $\rank\cA$ is the dimension of the linear space spanned by the normals to the hyperplanes of $\cA$.  For the coefficients of the characteristic polynomial it will be convenient to use the notation
\begin{equation}\label{eq:char_poly_coeff}
\chi_{\cA}(t) = \sum_{k=0}^d (-1)^{d-k} a_k t^k.
\end{equation}

\subsection{Main result}\label{subsec:result}
We are now ready to state a simplified version of our main result\footnote{After the first version of this paper has been uploaded to the arXiv, we have been informed by Giovanni Paolini and Davide Lofano that the same theorem has been obtained in their paper~\cite[Corollary~5.13]{lofano_paolini}. It seems that our proof is quite different (and more elementary).}.
\begin{theorem}\label{theo:main_simplified}
Let $\cA$ be an affine hyperplane arrangement in $\R^d$ whose characteristic polynomial $\chi_\cA(t)$ is written in the form~\eqref{eq:char_poly_coeff}. Take some $k\in \{0,\ldots,d\}$. Then,
$$
\# \{P\in \cR(\cA):  \dim (x,P) = k \} = a_k
$$
for every $x\in \R^d$ outside certain exceptional set which is a finite union of affine hyperplanes.
\end{theorem}

\begin{example}[First Zaslavsky's formula]
Let us show that the first Zaslavsky formula $\#\cR(\cA) = (-1)^d \chi_\cA(-1)$ is a consequence of Theorem~\ref{theo:main_simplified}. Take some point $x$ not belonging to the exceptional set. On the one hand, for every chamber $P\in \cR(\cA)$ there is a unique face whose relative interior contains the metric projection $\pi_P(x)$, hence interchanging the order of summation we get
$$
\sum_{k=0}^d \# \{P\in \cR(\cA):  \dim (x,P) = k \} = \sum_{P\in\cR(\cA)} \sum_{k=0}^d \ind_{\{\dim (x,P) = k\}} = \sum_{P\in\cR(\cA)} 1 = \# \cR(\cA).
$$
On the other hand, the sum on the left-hand side equals $\sum_{k=0}^d a_k$ by Theorem~\ref{theo:main_simplified}. Altogether, we arrive at $\# \cR(\cA) = \sum_{k=0}^d a_k$, which proves the first Zaslavsky formula. The second Zaslavsky formula will be discussed in Example~\ref{ex:zaslavsky_2}.
\end{example}

\begin{example}[Reflection arrangements]
Consider a finite reflection group $\cW$ acting on $\R^d$. Let  $\chi(t)$ be the characteristic polynomial of the associated reflection arrangement and $C$ one of its chambers.  Drton and Klivans~\cite[Conjecture~6]{drton_klivans} conjectured that for a ``generic'' point $x\in\R^d$ the number of group elements $g\in \cW$ with $\dim (gx, C) = k$ is equal to the absolute value of the coefficient of $t^k$ in $\chi(t)$, for all $k\in \{0,\ldots,d\}$. This conjecture is an easy consequence of Theorem~\ref{theo:main_simplified}. Indeed, since every $g\in \cW$ is an isometry, $\dim (gx, C)$  equals $\dim (x, g^{-1}C)$. If $g$ runs through all elements of $\cW$, then $g^{-1}C$ runs through all chambers of the reflection arrangement, and the conjecture follows from Theorem~\ref{theo:main_simplified}. Note that the characteristic polynomials of the reflection arrangements are known explicitly; see, e.g., \cite[p.~124]{bona_handbook}.
\end{example}

Let us now restate Theorem~\ref{theo:main_simplified} in a more explicit form involving a concrete description of the exceptional set.
First we need to define the notions of tangent and normal cones. Let
$$
\pos A:=\left\{\sum_{i=1}^m \lambda_i a_i:m\in\N,a_1,\dots,a_m\in A,\lambda_1,\dots,\lambda_m\ge 0\right\}.
$$
denote the \textit{positive hull} of a set $A\subset \R^d$.   The \emph{tangent cone} of a polyhedral set $P$ at its face $F\in\cF(P)$ is defined by
\begin{equation}\label{eq:T_F_P_def}
T_F(P) = \pos\{p-f_0: p\in P\}  = \{u\in \R^d: \exists \delta>0: f_0 + \delta u \in P\},
\end{equation}
where   $f_0$ is an arbitrary point in $\relint F$.
It is known that this definition does not depend on the choice of $f_0$ and that $T_F(P)$ is a polyhedral cone. Moreover, $T_F(P)$ contains the linear subspace $\aff F- f_0$, where $\aff F$ is the affine hull of $F$, i.e.\ the minimal affine subspace containing $F$.
For a polyhedral cone $C\subset \R^d$ its \emph{dual} or \emph{polar} cone is defined by
$$
C^\circ = \{z\in \R^d: \langle z,y \rangle \leq 0 \text{ for all } y\in C\},
$$
where $\langle\cdot, \cdot\rangle$ denotes the standard Euclidean scalar product on $\R^d$.
The \textit{normal cone} of a polyhedral set $P$ at its face $F\in \cF(P)$ is defined as the dual of the tangent cone:
$$
N_F(P) = (T_F(P))^\circ.
$$
By definition, $N_F(P)$ is a polyhedral cone contained in $(\aff F)^\bot$, the orthogonal complement of $\aff F$. Here, the orthogonal complement of an affine subspace $A\subset \R^d$ is the linear subspace
$$
A^\bot = \{z\in \R^d: \langle z,y \rangle = 0 \text{ for all } y\in A\}.
$$

Now, the metric projection of a point $x\in\R^d$ onto a polyhedral set $P$  satisfies $\pi_P(x) \in F$ for a face $F\in\cF(P)$ if and only if $x\in F+N_F(P)$. Here, $A+B = \{a+b: a\in A, b\in B\}$ is the Minkowski sum of the sets $A,B\subset \R^d$, which in our special case is even orthogonal meaning that every vector from $N_F(P)$ is orthogonal to every vector from $F$. Similarly, we have
$$
\pi_P(x)\in \relint F
\;\;\;
\Longleftrightarrow
\;\;\;
x\in (\relint F)+N_F(P).
$$

Let $\Int A$ denote the interior of a set $A$, and let $\partial A=A\backslash \Int A$ be the boundary of $A$. We are now ready to restate our main result in a more explicit form.

\begin{theorem}\label{theo:main}
Let $\cA$ be an affine hyperplane arrangement in $\R^d$ whose characteristic polynomial $\chi_\cA(t)$ is written in the form~\eqref{eq:char_poly_coeff}. Then, for every $k\in \{0,\ldots,d\}$  we have
\begin{equation}\label{eq:varphi_k_statement}
\varphi_k(x)
:=
\sum_{P\in \cR (\cA)} \sum_{F\in \cF_k(P)} \ind_{F+ N_F(P)} (x) = a_k,
\;\;\;
\text{ for all }
x\in \R^d\backslash E_k,
\end{equation}
where the exceptional set $E_k$ is given by
\begin{equation}\label{eq:E_k_def}
E_k
=
\bigcup_{P\in \cR(\cA)} \bigcup_{F\in \cF_k(P)} \partial (F + N_F(P)).
\end{equation}
Also, we have $\varphi_k(x) \geq a_k$ for every $x\in\R^d$.
\end{theorem}
An equivalent representation of the set  $E_k$, implying that it is a finite union of affine hyperplanes, will be given below; see~\eqref{eq:E'L_E''L_rep} and~\eqref{eq:exceptional_set_identity}.

\begin{example}
Let us consider a simple example showing that the exceptional set cannot be removed from the statement of Theorem~\ref{theo:main}. Consider an arrangement $\cA$ consisting of the coordinate axes $\{x_1=0\}$ and $\{x_2=0\}$ in $\R^2$. There are four chambers and the characteristic polynomial is given by $\chi_\cA(t) = (t-1)^2$.  It is easy to check that the functions $\varphi_0$ and $\varphi_1$ defined in~\eqref{eq:varphi_k_statement} are given by
$$
\varphi_0(x_1,x_2) = 1 + \ind_{\{x_1=0\}} + \ind_{\{x_2=0\}} + \ind_{\{x_1=x_2=0\}},
\qquad
\varphi_1(x_1,x_2) = 2\varphi_0 (x_1,x_2).
$$
These functions are strictly larger than $a_0=1$ and $a_1=2$ on the exceptional set $E_1=E_2 = \{x_1=0\}\cup\{x_2=0\}$.
\end{example}

\begin{remark}[Similar identities]
It is interesting to compare Theorem~\ref{theo:main} to the following identity: For every polyhedral set $P\subset \R^d$ we have
\begin{equation}\label{eq:identity_alternating}
\sum_{k=0}^d  \sum_{F\in\cF_k(P)}(-1)^{k} \ind_{F-N_F(P)}(x)
=
\begin{cases}
1, &\text{if $P$ is bounded},\\
0, &\text{if $P$ is unbounded and line-free,}
\end{cases}
\end{equation}
for all $x\in\R^d$, without an exceptional set.
Various versions of this formula valid outside certain exceptional sets of Lebesgue measure $0$ have been obtained starting with the work of McMullen~\cite[p.~249]{mcmullen}; see~\cite[Proof of Theorem~6.5.5]{SW08}, \cite[Hilfssatz 4.3.2]{glasauer_phd}, \cite{glasauer_local_steiner},
\cite[Corollary~2.25 on p.~89]{hug_habil}. The exceptional set has been removed independently in~\cite{schneider_combinatorial} (for polyhedral cones) and in~\cite{hug_kabluchko} (for general polyhedral sets). Cowan~\cite{cowan1} proved another identity for alternating sums of indicator functions of convex hulls. The exceptional set in Cowan's identity has been subsequently removed in~\cite{kabluchko_last_zaporozhets}.
\end{remark}

\subsection{Conic intrinsic volumes and characteristic polynomials}\label{subsec:conic_intrinsic}
As another consequence of our result we can re-derive a formula due to Klivans and Swartz~\cite[Theorem~5]{klivans_swartz} which expresses the coefficients of the characteristic polynomial of a linear hyperplane arrangement through the conic intrinsic volumes of its chambers. Let us first define conic intrinsic volumes; see~\cite[Section 6.5]{SW08} and~\cite{ALMT14,amelunxen_lotz} for more details.   Let $\xi$ be a random vector having an arbitrary rotationally invariant distribution on $\R^d$. As examples, one can think of the  uniform distribution  on the unit sphere in $\R^d$ or the standard normal distribution.  The $k$-th conic intrinsic volume $\nu_k(C)$ of a polyhedral cone $C\subset \R^d$ is defined as the probability that the metric projection of $\xi$ onto $C$ belongs to a relative interior of a $k$-dimensional face of $C$, that is
$$
\nu_k(C) = \P[\dim (\xi, C) = k],
\qquad
k\in \{0,\ldots,d\}.
$$
To state the formula of Klivans and Swartz~\cite[Theorem~5]{klivans_swartz}, consider a linear hyperplane arrangement, i.e.\ a finite collection $\cA=\{H_1,\ldots,H_m\}$ of hyperplanes in $\R^d$ passing through the origin.
The hyperplanes dissect $\R^d$ into finitely many polyhedral cones (the chambers of the arrangement). The formula of Klivans and Swartz~\cite[Theorem~5]{klivans_swartz} states that for every $k\in \{0,\ldots,d\}$ the sum of $\nu_k(C)$ over all chambers is equal to the absolute value of the $k$-th coefficient\footnote{Note that the characteristic polynomial in the notation of Klivans and Swartz~\cite[p.~420]{klivans_swartz}  corresponds to our $\chi_\cA(t)/t^{d-\rank(\cA)}$, where $\rank (\cA) = d - \dim \cap_{i=1}^m H_i$ is the rank of the arrangement.} of the characteristic polynomial $\chi_\cA(t)$, namely
\begin{equation}\label{eq:klivans_swartz}
\sum_{C\in \cR(\cA)} \nu_k(C) = a_k,
\;\;\;\text{ for all } \;\;\;
k\in \{0,\ldots,d\}.
\end{equation}
For proofs and extensions  of the Klivans-Swartz formula see~\cite[Theorem~4.1]{KVZ15}, \cite[Section~6]{amelunxen_lotz} and~\cite[Eq.~(15) and Theorem~1.2]{schneider_combinatorial}.
To see that~\eqref{eq:klivans_swartz} is a consequence of our results, note that by Theorem~\ref{theo:main_simplified} applied with $x$ replaced by $\xi$,
$$
\sum_{C\in \cR(\cA)} \ind_{\{\dim (\xi,C) = k\}} = a_k
\;\;\;
\text{ with probability }
1.
$$
Taking the expectation and interchanging it with the sum yields~\eqref{eq:klivans_swartz}. Thus, in the setting of linear arrangements, our result can be seen as an a.s.\ version of the Klivans-Swartz formula.

\begin{example}[Second Zaslavsky's formula]\label{ex:zaslavsky_2}
Let us use Theorem~\ref{theo:main} in combination with identity~\eqref{eq:identity_alternating} to sketch a proof of the second Zaslavsky formula for the number $b(\cA)$ of bounded chambers, namely  $b(\cA) = \sum_{k=0}^d (-1)^k a_k$, where we assumed for simplicity that the arrangement $\cA$ has full rank $d$ (the general case can be easily reduced to this one). Let $\xi$ be a random vector uniformly distributed on the unit sphere in $\R^d$. Applying identity~\eqref{eq:identity_alternating} to each chamber $P\in \cR(\cA)$ and the point $x = R \xi$, where $R>0$, and taking the expectation, we get
$$
\sum_{P\in \cR(\cA)} \sum_{k=0}^d (-1)^{k} \sum_{F\in\cF_k(P)} \P[R\xi \in F-N_F(P)] = b(\cA).
$$
It is an exercise to check that for every $F\in \cF(P)$,
$$
\lim_{R \to\infty} (\P[R\xi \in F-N_F(P)] - \P[R\xi \in F + N_F(P)]) = 0.
$$
It follows that
\begin{align*}
b(\cA)
&=
\lim_{R\to +\infty} \sum_{P\in \cR(\cA)} \sum_{k=0}^d (-1)^{k} \sum_{F\in\cF_k(P)} \P[R\xi \in F + N_F(P)] \\
&=
\lim_{R\to +\infty} \sum_{k=0}^d (-1)^{k} \E \sum_{P\in \cR(\cA)}  \sum_{F\in\cF_k(P)} \ind_{F + N_F(P)}(R\xi)\\
&=
\sum_{k=0}^d (-1)^k a_k
\end{align*}
by Theorem~\ref{theo:main}.
\end{example}

\subsection{Extension to \texorpdfstring{$j$}{j}-th level characteristic polynomials}
Let us finally  mention one simple extension of the above results. Let $\cL(\cA)$ be the set of all non-empty intersections of hyperplanes from $\cA$. By convention, the whole space $\R^d$ is also included in $\cL(\cA)$ as an intersection of the empty collection. Take some $j\in \{0,\ldots,d\}$ and let $\cL_j(\cA)$ denote the set of all $j$-dimensional affine subspaces in $\cL(\cA)$.
The \emph{restriction} of the arrangement $\cA$ to the subspace $L\in \cL(\cA)$ is defined as
$$
\cA^L= \{H\cap L: H\in \cA, H\cap L \neq L, H\cap L \neq \varnothing\},
$$
which is an affine hyperplane arrangement in the ambient space $L$. Note that it may happen that  $H_1\cap L = H_2\cap L$ for some different $H_1,H_2\in \cA$, in which case the corresponding hyperplane is listed just once in the arrangement $\cA^L$.

Now, the \emph{$j$-th level characteristic polynomial} of $\cA$ may be defined as
\begin{equation}\label{eq:chi_A_j_def}
\chi_{\cA,j}(t) = \sum_{L\in \cL_j(\cA)} \chi_{\cA^L}(t).
\end{equation}
We refer to~\cite[Section~2.4.1]{amelunxen_lotz} for this and other equivalent definitions. Note that in the case $j=d$ we recover the usual characteristic polynomial $\chi_\cA(t)$.   For the coefficients of the $j$-th level characteristic polynomial we use the notation
\begin{equation}\label{eq:chi_A_j_coeff}
\chi_{\cA,j}(t) = \sum_{k=0}^j (-1)^{j-k} a_{kj} t^k.
\end{equation}
Recall that $\cR(\cA)$ denotes the set of all closed chambers generated by the arrangement $\cA$. For $j\in \{0,\ldots,d\}$, let $\cR_j(\cA)$ be the set of all $j$-dimensional faces of all chambers, that is
$$
\cR_j(\cA) = \bigcup_{P\in \cR(\cA)} \cF_j(P).
$$
The $j$-th level extension of Theorem~\ref{theo:main} reads as follows.
\begin{theorem}\label{theo:main_j_level}
Let $\cA$ be an affine hyperplane arrangement in $\R^d$ whose $j$-th level characteristic polynomial $\chi_{\cA,j}(t)$ is written in the form~\eqref{eq:chi_A_j_coeff}. Then, for every $j\in \{0,\ldots,d\}$ and $k\in \{0,\ldots,j\}$  we have
$$
\sum_{P\in \cR_j (\cA)} \sum_{F\in \cF_k(P)} \ind_{F+ N_F(P)} (x) = a_{kj},
\;\;\;
\text{ for all }
x\in \R^d\backslash E_{kj},
$$
where the exceptional set $E_{kj}$ is given by
\begin{equation}\label{eq:E_kj_def}
E_{kj}
=
\bigcup_{P\in \cR_j(\cA)} \bigcup_{F\in \cF_k(P)} \partial (F + N_F(P)).
\end{equation}
\end{theorem}
\begin{proof}[Proof of Theorem~\ref{theo:main_j_level} assuming Theorem~\ref{theo:main}]
Consider any $L\in \cL_j(A)$ and apply Theorem~\ref{theo:main} to the hyperplane arrangement $\cA^L$ in the ambient space $L$. This yields
\begin{equation}\label{eq:j_level_proof1}
\sum_{\substack{P\in \cR_j (\cA):\\P\subset L}} \sum_{F\in \cF_k(P)} \ind_{F+ (N_F(P)\cap L_0)} (z) = a_{k,L},
\;\;\;
\text{ for all }
z\in L\backslash E_{k,L},
\end{equation}
where $L_0:=L-\pi_L(0)$ is a shift of the affine subspace $L$ that contains the origin, the $a_{k,L}$'s are defined by the formulae
\begin{equation}\label{eq:chi_A_L_coeff}
\chi_{\cA^L}(t) = \sum_{k=0}^j (-1)^{j-k} a_{k,L} t^k
\end{equation}
and  the exceptional sets $E_{k,L}\subset L$ are given by
\begin{equation}\label{eq:E_k_L_def}
E_{k,L}
=
\bigcup_{\substack{P\in \cR_j (\cA):\\P\subset L}} \bigcup_{F\in \cF_k(P)} \partial_L (F + (N_F(P)\cap L_0)).
\end{equation}
Here, $\partial_L$ denotes the boundary operator in the ambient space $L$.
Note that in~\eqref{eq:j_level_proof1} the normal cone of $F\in \cF_k(P)$ in the ambient space $L$ is represented as $N_F(P)\cap L_0$, where $N_F(P)$ denotes the normal cone in the ambient space $\R^d$.  Also,  we have the orthogonal sum decomposition $N_F(P) = (N_F(P) \cap L_0) + L^\bot$. Hence, we can rewrite~\eqref{eq:j_level_proof1} as
\begin{equation}\label{eq:j_level_proof2}
\sum_{\substack{P\in \cR_j (\cA):\\P\subset L}} \sum_{F\in \cF_k(P)} \ind_{F+ N_F(P)} (x) = a_{k,L},
\;\;\;
\text{ for all }
x\in \R^d \backslash (E_{k,L} + L^\bot).
\end{equation}

Since each $j$-dimensional face $P\in \cR_j(\cA)$ is contained in a unique affine subspace $L\in \cL_j(\cA)$, we can take the sum over all such $L$ arriving at
$$
\sum_{P\in \cR_j (\cA)} \sum_{F\in \cF_k(P)} \ind_{F+ N_F(P)} (x)
=
\sum_{L\in \cL_j(\cA)} \sum_{\substack{P\in \cR_j (\cA):\\P\subset L}} \sum_{F\in \cF_k(P)} \ind_{F+ N_F(P)} (x)
=
\sum_{L\in \cL_j(\cA)} a_{k,L}
$$
for all $x\in\R^d$ outside the following exceptional set:
\begin{align*}
\bigcup_{L\in \cL_j(\cA)} (E_{k,L} + L^\bot)
&=
\bigcup_{L\in \cL_j(\cA)} \bigcup_{\substack{P\in \cR_j (\cA):\\P\subset L}} \bigcup_{F\in \cF_k(P)}
\left(\partial_L (F + (N_F(P)\cap L_0)) + L^\bot\right)\\
&=
\bigcup_{P\in \cR_j (\cA)} \bigcup_{F\in \cF_k(P)}
\partial (F + N_F(P))
= E_{kj}.
\end{align*}
Here, we used that $\partial_L(A) + L^\bot = \partial (A+L^\bot)$ for every set $A\subset L$.
It follows from~\eqref{eq:chi_A_j_def}, \eqref{eq:chi_A_j_coeff}, \eqref{eq:chi_A_L_coeff} that
$$
\sum_{L\in \cL_j(\cA)} a_{k,L} = a_{kj},
$$
which completes the proof.
\end{proof}

\begin{remark}
Using almost the same argument as in Section~\ref{subsec:conic_intrinsic}, Theorem~\ref{theo:main_j_level} yields the following $j$-th level extension of the Klivans-Swartz formula obtained in~\cite[Theorem~6.1]{amelunxen_lotz} and~\cite[Eq.~(15)]{schneider_combinatorial}:
\begin{equation}\label{eq:klivans_swartz_extended}
\sum_{P\in \cR_j(\cA)} \nu_k(P) = a_{kj},
\;\;\;\text{ for all } \;\;\;
j\in \{0,\ldots,d\},
\;\;
k\in \{0,\ldots,j\}.
\end{equation}
\end{remark}

The remaining part of this paper is devoted to the proof of Theorem~\ref{theo:main}.

\section{Proof of Theorem~\ref{theo:main}}
\noindent
\textit{Step 1.} We start with a proposition which, as we shall see in Remark~\ref{rem:k_0} below, implies Theorem~\ref{theo:main} for linear hyperplane arrangements in the special case $k=0$. Recall
that the dual cone of  a polyhedral cone $C\subset \R^d$  is defined by
$$
C^\circ = \{z\in \R^d: \langle z,y \rangle \leq 0 \text{ for all } y\in C\}.
$$
It is known that $C^{\circ\circ} = C$; see~\cite[Prop.~2.3]{amelunxen_lotz}. The \emph{lineality space} of a cone $C$ is the largest linear space contained in $C$ and is explicitly given by $C\cap (-C)$. It is known that the linear space spanned by the dual cone $C^\circ$ coincides with the orthogonal complement of the lineality space of $C$; see, e.g., \cite[Prop.~2.5]{amelunxen_lotz} for a more general statement. In particular, the lineality space of $C$ is trivial (i.e.\ equal to $\{0\}$) if and only if $C^\circ$ has  non-empty interior.

\begin{proposition}\label{prop:k_0}
Let $\cA$ be a linear hyperplane arrangement in $\R^d$.  Then,
\begin{equation}\label{eq:prop_k_0}
\sum_{C\in \cR(\cA)} \ind_{C^\circ}(x) = a_0,
\;\;\;
\text{ for all }
x\in \R^d \backslash E_0^*,
\end{equation}
where $a_0$ is defined by~\eqref{eq:char_poly_def} and~\eqref{eq:char_poly_coeff} and the exceptional set $E_0^*$ is given by
\begin{equation}\label{eq:E_0*_def}
E_0^* := \bigcup_{L\in \cL(\cA)\backslash\{0\}} L^\bot = \bigcup_{C\in \cR(\cA)} \partial (C^\circ).
\end{equation}
\end{proposition}
\begin{proof}
Since for linear arrangements $C\mapsto -C$ defines a bijective self-map of $\cR(\cA)$ and since $(-C)^\circ = -(C^\circ)$, we have
$$
\sum_{C\in \cR(\cA)} \ind_{C^\circ}(x) = \sum_{C\in \cR(\cA)} \ind_{-C^\circ}(x),
\;\;\;
\text{ for all }
x\in \R^d.
$$
Therefore, it suffices to prove that
$$
\sum_{C\in \cR(\cA)} (\ind_{C^\circ}(x) + \ind_{-C^\circ}(x))  = 2a_0,
\;\;\;
\text{ for all }
x\in \R^d\backslash E_0^*.
$$
Since every cone $C\in \cR(\cA)$ is full-dimensional, implying that the dual cone has trivial lineality space $C^\circ \cap (-C^\circ) = \{0\}$, it suffices to prove that
$$
\sum_{C\in \cR(\cA)} \ind_{C^\circ \cup -C^\circ}(x)  = 2a_0,
\;\;\;
\text{ for all }
x\in \R^d\backslash E_0^*.
$$
Let $L(x):= \{\lambda x: \lambda \in\R\}$ be the $1$-dimensional line generated by $x\in \R^d\backslash\{0\}$. Then, $x\in C^\circ \cup -C^\circ$ if and only if $L\cap C^\circ \neq \{0\}$. It therefore suffices to prove that
$$
\sum_{C\in \cR(\cA)} \ind_{\{L(x) \cap C^\circ \neq \{0\}\}}  = 2a_0,
\;\;\;
\text{ for all }
x\in \R^d\backslash E_0^*.
$$
In a slightly different form, this result is contained in~\cite[Thm.~1.2, Eq.~(16)]{schneider_combinatorial}. For completeness, we provide a proof.
By the Farkas lemma~\cite[Lemma~2.4]{amelunxen_lotz}, $L(x) \cap C^\circ \neq \{0\}$ is equivalent to $L(x)^\bot \cap \Int C = \varnothing$. Thus, we need to show that
\begin{equation}\label{eq:tech1}
\sum_{C\in \cR(\cA)} \ind_{\{L(x)^\bot  \cap \Int C  = \varnothing\}}  = 2a_0,
\;\;\;
\text{ for all }
x\in \R^d\backslash E_0^*.
\end{equation}
By the first Zaslavsky formula, the total number of chambers of $\cA$ is given by $\# \cR(\cA) = (-1)^d \chi_\cA(-1) = \sum_{k=0}^d a_k$. By the second Zaslavsky formula, $\chi_\cA(1)=0$ (because there are no bounded chambers in a linear arrangement). Hence, $\sum_{k=0}^d (-1)^k a_k = 0$ and  it follows that $\# \cR(\cA) = 2 \sum_{k=0}^{[d/2]}a_{2k}$. In view of this, it suffices to show that
\begin{equation}\label{eq:tech2}
\sum_{C\in \cR(\cA)} \ind_{\{L(x)^\bot  \cap \Int C  \neq  \varnothing\}}  = 2 \sum_{k=1}^{[d/2]}a_{2k},
\;\;\;
\text{ for all }
x\in \R^d\backslash E_0^*.
\end{equation}
This identity is known~\cite[Thm.~3.3]{KVZ15} provided that the hyperplane $L(x)^\bot$ is in general position with respect to the arrangement $\cA$. By definition~\cite[Sec.~3.1]{KVZ15} , the general position condition means that for every $L\in \cL(\cA)$ with $L\neq \{0\}$, we have $\dim (L\cap L(x)^\bot) = \dim L-1$. This is the same as to require that  $L$ is not a subset of $L(x)^\bot$ or, equivalently, that $x\notin L^\bot$. So, the above identity holds for all $x\in \R^d \backslash \cup_{L\in \cL(\cA)\backslash\{0\}} (L^\bot)$, which completes the proof of~\eqref{eq:prop_k_0}.  The second representation of the exceptional set $E_0^*$ in~\eqref{eq:E_0*_def} was mentioned just for completeness. We shall prove it in Lemma~\ref{lem:boundary_normal_cones} without using it before.
\end{proof}

\begin{lemma}\label{lem:non_gen_pos_remark}
Let $\cA$ be a linear hyperplane arrangement in $\R^d$.  Then,
\begin{equation*}
\sum_{C\in \cR(\cA)} \ind_{C^\circ}(x) \geq  a_0,
\;\;\;
\text{ for all }
x\in \R^d.
\end{equation*}
\end{lemma}
\begin{proof}
The proof of Proposition~\ref{prop:k_0} applies with minimal modifications.  Indeed, by~\cite[Lemma 3.5, Eq.~(38)]{KVZ15} (note that $\cR(\cA)$ denotes the collection of \emph{open} chambers there), the equality in~\eqref{eq:tech2} has to be replaced by the inequality $\leq$, which means that the equality in~\eqref{eq:tech1}  should be replaced by $\geq$. The rest of the proof applies.
\end{proof}

\begin{remark}\label{rem:k_0}
With Proposition~\ref{prop:k_0} at hand, we can prove Theorem~\ref{theo:main} for $k=0$ provided the arrangement $\cA$ is linear. Assume first that $\cA$ is essential, i.e.\ it has full rank meaning that $\cap_{H\in \cA} H = \{0\}$. Then, $F=\{0\}$ is the only $0$-dimensional face of every chamber $C\in \cR(\cA)$.  
The normal cone of $C$ an this face is $N_{\{0\}}(C) = C^\circ$.
Hence, the case $k=0$ of Theorem~\ref{theo:main} follows from Proposition~\ref{prop:k_0} and Lemma~\ref{lem:non_gen_pos_remark}.
In the case of a non-essential linear arrangement, that is  if $L_* := \cap_{H\in \cA} H \neq  \{0\}$, Theorem~\ref{theo:main} becomes trivial for $k=0$ as there are no $0$-dimensional faces and the zeroth coefficient of $\chi_\cA(t)$ vanishes by its definition~\eqref{eq:char_poly_def}. Proposition~\ref{prop:k_0} also becomes trivial since the dual cone $C^\circ$ of every chamber $C$ is contained in $L_*^\bot$, which coincides with the exceptional set $\cup_{L\in \cL(\cA)\backslash\{0\}} (L^\bot)$. Since $a_0=0$, both sides of~\eqref{eq:prop_k_0} vanish for $x\notin L_*^\bot$.
\end{remark}

\begin{remark}
In the special case of reflection arrangements, Proposition~\ref{prop:k_0} can be found in the paper of Denham~\cite[Thm.2]{denham}; see also~\cite{deconcini} for a related work.
\end{remark}

\vspace*{2mm}
\noindent
\textit{Step 2.}  We are interested in the function
$$
\varphi_k(x) = \sum_{P\in \cR (\cA)} \sum_{F\in \cF_k(P)} \ind_{F+ N_F(P)} (x), \qquad x\in\R^d.
$$
First of all note that in the case $k=d$ we trivially have $\varphi_d(x) = 1$ for all $x\in \R^d \backslash \cup_{H\in \cA} H$.

In the following, fix some $k\in \{0,\ldots,d-1\}$. Recall that $\cR_j(\cA)= \cup_{P\in  \cR(\cA)} \cF_j(P)$ is the set of all $j$-dimensional faces of all chambers of $\cA$ (without repetitions).   Interchanging the order of summation, we may write
$$
\varphi_k(x) =  \sum_{F\in \cR_k(\cA)} \sum_{\substack{P\in \cR (\cA):\\ F\in \cF_k(P)}} \ind_{F+ N_F(P)} (x).
$$
Recall also that $\cL(\cA)$ is the set of all non-empty intersections of hyperplanes from $\cA$ and that $\cL_k(\cA)$ is the set of all $k$-dimensional affine subspaces in $\cL(\cA)$.
Since each $k$-dimensional face $F\in \cR_k(\cA)$ is contained in a unique $k$-dimensional affine subspace $L\in \cL_k(\cA)$, we may split the sum in the above formula for $\varphi_k(x)$ as follows:
\begin{equation}\label{eq:varphi_k_sum}
\varphi_k(x) =  \sum_{L\in \cL_k(\cA)}  \varphi_L(x),
\end{equation}
where for each $L\in \cL_k(\cA)$ we define
\begin{equation}\label{eq:varphi_L_def}
\varphi_L(x)  = \sum_{\substack{F\in \cR_k(\cA):\\ F\subset L}} \sum_{\substack{P\in \cR (\cA):\\ F\in \cF_k(P)}} \ind_{F+ N_F(P)} (x).
\end{equation}

\vspace*{2mm}
\noindent
\textit{Step 3.} In this step we shall prove that for every $k\in \{0,\ldots, d-1\}$ and every $L\in \cL_k(\cA)$ the function $\varphi_L(x)$ defined in~\eqref{eq:varphi_L_def} is constant outside the exceptional set
\begin{equation}\label{eq:E_L_def}
E(L) := E'(L) \cup E''(L),
\end{equation}
where
\begin{equation}\label{eq:E_L_def_cont}
E'(L) := \bigcup_{\substack{L_{k-1}\in \cL_{k-1}(\cA):\\L_{k-1}\subset L}} (L_{k-1} + L^\bot)
\;\;\;
\text{ and }
\;\;\;
E''(L) := \bigcup_{\substack{L_{k+1}\in \cL_{k+1}(\cA):\\L_{k+1}\supset L}} (L_{k+1}^\bot + L).
\end{equation}
For $k=0$ we put $E'(L) := \varnothing$. Note that $E(L)$ is a finite union of affine hyperplanes. Moreover, we shall identify the value of the constant in terms of the characteristic polynomial of some hyperplane arrangement in $L^\bot$,  the orthogonal complement of $L$. The final result will be stated in Proposition~\ref{prop:varphi_L} at the end of this step.


First we need to introduce some notation. Recall that  $\langle\cdot, \cdot\rangle$ denotes the standard Euclidean scalar product on $\R^d$. Let the affine hyperplanes $H_1,\ldots,H_m$ constituting the arrangement $\cA$ be given by the equations
$$
H_i= \{z\in\R^d: \langle z, y_i\rangle = c_i\} , \qquad i \in \{1,\ldots,m\},
$$
for some vectors $y_1,\ldots,y_m \in \R^d\backslash\{0\}$ and some scalars $c_1,\ldots, c_m\in \R$. Every closed chamber of the arrangement $\cA$ can be represented in the form
$$
P = \{z\in\R^d: \eps_1 (\langle z, y_1\rangle - c_1) \leq 0, \ldots, \eps_m (\langle z, y_m\rangle - c_m) \leq 0\}
$$
with  a suitable choice of $\eps_1,\ldots,\eps_m\in \{-1,+1\}$.
Conversely, every set of the above form defines a closed chamber provided its \emph{interior is non-empty}. Note in passing that the interior of this chamber is represented by the corresponding  strict inequalities as follows:
$$
\Int P = \{z\in\R^d: \eps_1 (\langle z, y_1\rangle - c_1) < 0, \ldots, \eps_m (\langle z, y_m\rangle - c_m) < 0\}.
$$
Finally, the chambers determined by  two different tuples $(\eps_1,\ldots,\eps_m)$ and $(\eps_1',\ldots,\eps_m')$ have disjoint interiors. Indeed, if the tuples differ in the $i$-th component, then any point $z$ in the relative interior of one chamber satisfies $\langle z, y_i\rangle < c_i$, whereas the points in the relative interior of the other chamber satisfy the converse inequality.

Fix some $k$-dimensional affine subspace $L\in \cL_k(\cA)$, where $k\in \{0,\ldots,d-1\}$.  It can be written in the form
$$
L = \{z\in\R^d:  \langle z, y_{i}\rangle  = c_{i} \text{ for all } i\in I \}
$$
for a suitable subset $I\subset \{1,\ldots,m\}$. Without restriction of generality we may assume that $L$ passes through the origin (otherwise we could translate everything). It follows that $c_i=0$ for $i\in I$.
Moreover, after renumbering (if necessary) the hyperplanes and their defining equations, we may assume that the linear subspace $L$ is given by the equations
\begin{equation}\label{eq:L_def}
L = \{z\in \R^d: \langle z, y_1\rangle = 0, \ldots, \langle z, y_\ell\rangle = 0\}
\end{equation}
for some $\ell\in \{d-k,\ldots, m\}$.
Finally, without loss of generality we may assume that $H_i\cap L$ is a strict subset of $L$ for all $i\in \{\ell+1,\ldots,m\}$ since otherwise we could include the defining equation of $H_i$ into the list on the right-hand side of~\eqref{eq:L_def}.

Take any point $x\in \R^d\backslash E(L)$, where we recall that $E(L)$ is defined by~\eqref{eq:E_L_def} and~\eqref{eq:E_L_def_cont}.  The orthogonal projection of $x$ onto the linear subspace $L$, denoted by $\pi_L(x)$, is contained in the relative interior of some uniquely defined face $G\in \cup_{p=0}^k \cR_p (\cA)$ with $G\subset L$. In fact, we even have $G\in \cR_k (\cA)$ because if the dimension of $G$ would be strictly smaller than $k$, we could find some $L_{k-1}\in \cL_{k-1}(\cA)$ with $G\subset L_{k-1}\subset L$. This would contradict the assumption $x\notin E'(L)$. So, we have
$$
\pi_L(x) \in \relint G, \qquad G\in \cR_k (\cA),  \qquad G\subset L.
$$
Then, the definition of $\varphi_L(x)$ given in~\eqref{eq:varphi_L_def} simplifies as follows:
\begin{equation}\label{eq:varphi_L_rep1}
\varphi_L(x) =  \sum_{\substack{P\in \cR (\cA):\\ G\in \cF_k(P)}} \ind_{G+ N_G(P)} (x).
\end{equation}
Indeed, for every $F\in \cR_k(\cA)$ and $P\in \cR(\cA)$ with $F\subset L$, $F\in \cF_k(P)$ and $F\neq G$ we have $x\notin F + N_F(P)$, which follows from the fact that $x\in \relint G + L^\bot$, while $\relint G \cap F = \varnothing$ and $N_F(P)\subset L^\bot$. This means that all terms with $F\neq G$ do not contribute to the right-hand side of~\eqref{eq:varphi_L_def}.

By changing, if necessary,  the signs of some $y_i$'s and the corresponding $c_i$'s, we may assume that the face $G$ is given as follows:
\begin{align}
G
&=
\{z\in L: \langle z, y_{\ell+1}\rangle \leq  c_{\ell+1},\ldots, \langle z, y_{m}\rangle \leq  c_{m}\}\notag \\
&=
\{z\in \R^d: \langle z, y_1\rangle = 0, \ldots, \langle z, y_\ell\rangle = 0, \langle z, y_{\ell+1}\rangle \leq  c_{\ell+1},\ldots, \langle z, y_{m}\rangle \leq  c_{m}\}. \label{eq:G_eqs}
\end{align}
The relative interior of $G$ is given by the following strict inequalities:
\begin{align}
\relint G
&=
\{z\in L: \langle z, y_{\ell+1}\rangle <  c_{\ell+1},\ldots, \langle z, y_{m}\rangle <  c_{m}\}\notag \\
&=
\{z\in \R^d: \langle z, y_1\rangle = 0, \ldots, \langle z, y_\ell\rangle = 0, \langle z, y_{\ell+1}\rangle <  c_{\ell+1},\ldots, \langle z, y_{m}\rangle <  c_{m}\}. \label{eq:relint_G_eqs}
\end{align}
Let now $P\in \cR(\cA)$ be a closed chamber such that $G\in \cF_k(P)$. Then, there exist some $\eps_1,\ldots,\eps_\ell\in \{-1,+1\}$ such that $P$ is given by
\begin{multline}\label{eq:P_eps_def}
P
=
P_{\eps_1,\ldots,\eps_\ell}:=
\{z\in \R^d: \eps_1 \langle z, y_{1}\rangle \leq 0,\ldots, \eps_\ell \langle z, y_{\ell}\rangle \leq 0,\\
 \langle z, y_{\ell+1}\rangle \leq   c_{\ell+1},\ldots, \langle z, y_{m}\rangle \leq  c_{m}\}.
\end{multline}
Conversely, if for some $\eps_1,\ldots,\eps_\ell\in \{-1,+1\}$ the interior of the set $P_{\eps_1,\ldots,\eps_\ell}$ defined above is non-empty, then $P_{\eps_1,\ldots,\eps_\ell}$ is a chamber in $\cR(\cA)$ and it contains $G$ as a $k$-dimensional face. Hence, we can rewrite~\eqref{eq:varphi_L_rep1} as follows:
$$
\varphi_L(x)  = \sum_{\substack{\eps_1,\ldots,\eps_\ell\in \{-1,+1\}:\\\Int P_{\eps_1,\ldots,\eps_\ell}\neq \varnothing}} \ind_{G + N_G(P_{\eps_1,\ldots,\eps_\ell})}(x).
$$
Write $x=\pi_L(x)+\pi_{L^\bot}(x)$ as a sum of its orthogonal projections $\pi_L(x)$ and $\pi_{L^\bot}(x)$ onto $L$ and $L^\bot$, respectively. Since $\pi_L(x)\in G\subset L$ and $N_G(P_{\eps_1,\ldots,\eps_\ell}) \subset L^\bot$, we arrive at
\begin{equation}\label{eq:varphi_L_formula}
\varphi_L(x)
=
\sum_{\substack{\eps_1,\ldots,\eps_\ell\in \{-1,+1\}:\\\Int P_{\eps_1,\ldots,\eps_\ell}\neq \varnothing}} \ind_{N_G(P_{\eps_1,\ldots,\eps_\ell})}\left(\pi_{L^\bot}(x)\right).
\end{equation}

Let us now characterize first the  tangent and then the normal cone of the face $G$ in the polyhedral set $P_{\eps_1,\ldots, \eps_\ell}$. Take some $z_0 \in \relint G$. Then, by~\eqref{eq:relint_G_eqs},
\begin{equation}\label{eq:relint_G_eqs_rep}
\langle z_0, y_1\rangle = 0, \; \ldots,\; \langle z_0, y_\ell\rangle = 0, \;
\langle z_0, y_{\ell+1}\rangle < c_{\ell+1}, \; \ldots, \;  \langle z_0, y_{m}\rangle <  c_{m}.
\end{equation}
By definition, see~\eqref{eq:T_F_P_def}, the tangent cone is given by
$$
T_G(P_{\eps_1,\ldots, \eps_\ell}) = \{u\in \R^d: \exists \delta>0: z_0 + \delta u \in P_{\eps_1,\ldots,\eps_\ell}\}.
$$
It follows from this definition together with~\eqref{eq:P_eps_def} and~\eqref{eq:relint_G_eqs_rep}  that
$$
T_G(P_{\eps_1,\ldots, \eps_\ell}) = \{u\in \R^d:  \langle u, \eps_1 y_{1}\rangle \leq 0,\ldots,  \langle u, \eps_\ell y_{\ell}\rangle \leq 0\}.
$$
Note that the linear span of $y_1,\ldots,y_\ell$ is $L^\bot$ by~\eqref{eq:L_def}. The tangent cone $T_G(P_{\eps_1,\ldots, \eps_\ell})$ contains the linear space $L$.
Let us now restrict our attention to the space $L^\bot$ and define the cone
\begin{equation}\label{eq:T_eps_def}
T_{\eps_1,\ldots,\eps_l}
:=
\{u\in L^\bot:  \langle u, \eps_1 y_{1}\rangle \leq 0,\ldots,  \langle u, \eps_\ell y_{\ell}\rangle \leq 0\}
=
T_G(P_{\eps_1,\ldots, \eps_\ell}) \cap L^\bot
\subset
L^\bot.
\end{equation}
Then, the tangent cone  $T_G(P_{\eps_1,\ldots, \eps_\ell})$ can be represented as the direct orthogonal sum
$$
T_G(P_{\eps_1,\ldots, \eps_\ell}) = L + T_{\eps_1,\ldots,\eps_l}, \qquad T_{\eps_1,\ldots,\eps_l} \subset L^\bot.
$$
Taking the polar cone, we obtain the normal cone of the face $G$ in the polyhedral set $P_{\eps_1,\ldots,\eps_\ell}$:
\begin{equation}\label{eq:N_G_equals_T_circ}
N_G(P_{\eps_1,\ldots, \eps_\ell})  = L^\bot \cap T_{\eps_1,\ldots,\eps_l}^\circ.
\end{equation}
That is, $N_G(P_{\eps_1,\ldots, \eps_\ell})$ is just the dual cone of $T_{\eps_1,\ldots,\eps_l}$ taken with respect to the ambient space $L^\bot$. Although we shall not use this fact in the sequel, let us mention that the normal cone can be represented as the positive hull
$$
N_G(P_{\eps_1,\ldots, \eps_\ell})  = \pos (\eps_1 y_1,\ldots, \eps_\ell y_\ell) = \{\lambda_1 \eps_1 y_1 + \ldots + \lambda_\ell \eps_\ell y_\ell: \lambda_1,\ldots,\lambda_\ell\geq 0\}.
$$

In the following, we shall argue that those cones of the form $T_{\eps_1,\ldots,\eps_l}$ that have non-empty interior are the chambers of certain linear hyperplane arrangement $\cA(L)$ in $L^\bot$. The dual cones of these chambers are the normal cones $N_G(P_{\eps_1,\ldots, \eps_\ell})$. It is crucial that this arrangement is completely determined by $y_1,\ldots,y_\ell$ and does not depend on $G\subset L$. Applying Proposition~\ref{prop:k_0}, we shall prove that $\varphi_L(x)$ is constant outside some explicit exceptional Lebesgue null set.

Let us be more precise. First of all, note that the vectors $y_1,\ldots,y_\ell$ are pairwise different. Indeed, if two of them would be equal, say $y_1=y_2$, then  (in view of $c_1=c_2=0$) the corresponding hyperplanes $H_1$ and $H_2$ would be equal, which is prohibited by the definition of the hyperplane arrangement. Therefore, the orthogonal complements of the vectors $y_1,\ldots, y_\ell$ (taken with respect to the ambient space $L^\bot$) are also pairwise different and define a linear hyperplane arrangement in $L^\bot$ which we denote by
\begin{equation}\label{eq:A_L_def}
\cA(L) := \{L^\bot \cap y_1^\bot,\ldots, L^\bot \cap y_\ell^\bot\}.
\end{equation}
Since the linear span of $y_1,\ldots,y_\ell$ is $L^\bot$ by~\eqref{eq:L_def},  this arrangement is essential, that is the intersection of its hyperplanes is $\{0\}$. The chambers of the arrangement $\cA(L)$ are those of the cones $T_{\eps_1,\ldots,\eps_\ell}$, $(\eps_1,\ldots,\eps_\ell)\in \{-1,+1\}^\ell$, defined in~\eqref{eq:T_eps_def}, that have non-empty interior in $L^\bot$. Note also that $\cA(L)$ is uniquely determined by  the choice of $L\in \cL_k(\cA)$ and does not depend on $G$.

Now we claim that for $(\eps_1,\ldots,\eps_\ell)\in \{-1,+1\}^\ell$ the relative interior of the cone $T_{\eps_1,\ldots,\eps_\ell}$ is non-empty if and only if the interior of the polyhedral set $P_{\eps_1,\ldots, \eps_\ell}$ is non-empty.
If $\Int P_{\eps_1,\ldots, \eps_\ell}$ is non-empty, then it has dimension $d$,  $G\in \cF_k(P)$, and the tangent cone $T_G(P_{\eps_1,\ldots,\eps_\ell})$ is strictly larger than the linear space $L$ (because the latter has dimension $k<d$). It follows from~\eqref{eq:T_eps_def} that $\relint T_{\eps_1,\ldots,\eps_\ell}\neq \varnothing$. Conversely, if $\relint T_{\eps_1,\ldots,\eps_\ell}\neq \varnothing$, then $T_{\eps_1,\ldots,\eps_\ell}$ has the same dimension as $L^\bot$, while $G$ has the same dimension as $L$. It follows that the dimension of $P_{\eps_1,\ldots,\eps_\ell}$ is $d$, thus its interior is non-empty.

From the above it follows that the formula for the function $\varphi_L$ stated in~\eqref{eq:varphi_L_formula} can be written as the following sum over the chambers of the arrangement $\cA(L)$:
\begin{equation}\label{eq:varphi_L_formula2}
\varphi_L(x)
=
\sum_{C\in \cR(\cA(L))} \ind_{C^\circ}(\pi_{L^\bot}(x)).
\end{equation}

We are now going to apply Proposition~\ref{prop:k_0} to the hyperplane arrangement $\cA(L)$ in  the ambient space $L^\bot$. This is possible provided $\pi_{L^\bot}(x)$ does not belong to the exceptional set $E_0^*$ defined in Proposition~\ref{prop:k_0}. In our setting of the ambient space $L^\bot$, the exceptional set is given by
\begin{equation}\label{eq:E_0*_def2}
E_0^*
=
\bigcup_{M\in \cL(\cA(L))\backslash\{0\}} (M^\bot \cap L^\bot).
\end{equation}
Each linear subspace $M\in \cL(\cA(L))\backslash\{0\}$  has the form $M= (\cap_{i\in I} y_i^\bot) \cap L^\bot$ for some set $I\subset \{1,\ldots,\ell\}$. Then,  the corresponding orthogonal complement $M^\bot \cap L^\bot$ has the form $\lin \{y_i: i\in I\}$, where $\lin A$ denotes the linear subspace spanned by the set $A$. Moreover, the condition $M\neq \{0\}$ is equivalent to  the condition  $\lin \{y_i: i\in I\} \neq L^\bot$. Since the linear span of the vectors $y_1,\ldots,y_\ell$ is $L^\bot$ by~\eqref{eq:L_def},  any linear subspace of the form $\lin \{y_i: i\in I\}\neq L^\bot$  is contained in a linear subspace of the form $\lin \{y_i: i\in I'\}$, for some $I'\subset \{1,\ldots,\ell\}$ satisfying the following condition:
\begin{equation}\label{eq:cond_I'}
\dim\lin \{y_i: i\in I'\} = \# I' = \dim L^\bot - 1= d-k-1.
\end{equation}
Therefore, we have
$$
E_0^*
=
\bigcup_{\substack{I'\subset \{1,\ldots,\ell\}:\\{\text{\eqref{eq:cond_I'} holds}}}} \lin \{y_i: i\in I'\}.
$$
Given $I'\subset \{1,\ldots,\ell\}$ such that~\eqref{eq:cond_I'} holds, define the linear subspace
\begin{equation}\label{eq:L_k+1_representation}
L_{k+1}:=\{z\in\R^d: \langle z, y_i \rangle = 0 \text{ for all } i\in I'\}\subset \R^d.
\end{equation}
Then, $L_{k+1}$ is non-empty since $L \subset L_{k+1}$ and, moreover,  the dimension of $L_{k+1}$ equals  $k+1$, that is  $L_{k+1}\in \cL_{k+1}(\cA)$ (recall that the case $k=d$ has been excluded from the very beginning).   Conversely, every $L_{k+1}\in \cL_{k+1}(\cA)$ containing $L$ can be represented in the form~\eqref{eq:L_k+1_representation} with some $I'\subset \{1,\ldots,\ell\}$ satisfying~\eqref{eq:cond_I'}. Taking into account that $\lin \{y_i: i\in I'\} = L_{k+1}^\bot$, it follows that
\begin{equation}\label{eq:E_0*}
E_0^*
=
\bigcup_{\substack{L_{k+1}\in \cL_{k+1}(\cA):\\  L_{k+1}\supset L}} L_{k+1}^\bot\subset L^\bot.
\end{equation}
Proposition~\ref{prop:k_0} applies to all $x\in\R^d$ such that $\pi_{L^\bot}(x) \notin E_0^*$. This is equivalent to the condition that $x$ is outside the set
$$
\bigcup_{\substack{L_{k+1}\in \cL_{k+1}(\cA):\\ L_{k+1} \supset L}} (L_{k+1}^\bot + L),
$$
which coincides with the set $E''(L)$ introduced in~\eqref{eq:E_L_def_cont}.

Applying Proposition~\ref{prop:k_0} and Lemma~\ref{lem:non_gen_pos_remark} with the ambient space $L^\bot$ to the right-hand side of~\eqref{eq:varphi_L_formula2}, we arrive at the following result.
\begin{proposition}\label{prop:varphi_L}
Let $\cA$ be an affine hyperplane arrangement in $\R^d$. Fix some $k\in \{0,\ldots,d-1\}$ and $L\in \cL_k(\cA)$. Then, the function $\varphi_L$ defined in~\eqref{eq:varphi_L_def} satisfies
$$
\varphi_L(x) = a_0(L),
\qquad
\text{ for every }
x\in  \R^d\backslash E(L),
$$
where the exceptional set $E(L)$ is given by~\eqref{eq:E_L_def} and~\eqref{eq:E_L_def_cont},  and $a_0(L)$ is $(-1)^{d-k}$ times the zeroth coefficient of the characteristic polynomial of the linear arrangement $\cA(L)$ in $L^\bot$ defined by~\eqref{eq:A_L_def}. Also, for all $x\in\R^d$ we have  $\varphi_L(x)\geq a_0(L)$.
\end{proposition}


\vspace*{2mm}
\noindent
\textit{Step 4.}
Let us first write down a more explicit expression for $a_0(L)$ appearing in Proposition~\ref{prop:varphi_L}.
Recalling the definition of the characteristic polynomial, see~\eqref{eq:char_poly_def}, we can write
$$
\chi_{\cA(L)} (t) =  \sum_{J\subset \{1,\ldots,\ell\}} (-1)^{\# J} t^{\dim L^\bot - \rank \{y_j: j\in J\}},
$$
where $\rank \{y_j: j\in J\}$ denotes the dimension of the linear span of a system of vectors $\{y_j: j\in J\}$. Taking the zeroth coefficient of this polynomial and multiplying it with $(-1)^{d-k}$,  we can write Proposition~\ref{prop:varphi_L} as follows:
$$
\varphi_L(x) = a_0(L)
=
(-1)^{d-k} \sum_{\substack{J\subset \{1,\ldots,\ell\}:\\ \lin\{y_j: j\in J\} = L^\bot}}
(-1)^{\# J},
\;\;\;
\text{ for all }
x\in \R^d\backslash E(L).
$$
Recalling the representation of $L$ stated in~\eqref{eq:L_def}, we see that a set of vectors $\{y_j: j\in J\}$ with $J\subset \{1,\ldots,\ell\}$  contributes to the above sum if and only if $L = \cap_{j\in J} H_j$. Moreover, a set $J\subset \{1,\ldots,m\}$ which is not completely contained in $\{1,\ldots,\ell\}$ cannot satisfy $L = \cap_{j\in J} H_j$ since $H_j\cap L$ is a strict subset of $L$ for all $j\in \{\ell+1,\ldots,m\}$; see the discussion after~\eqref{eq:L_def}.   Therefore, we can rewrite the above sum as follows:
$$
\varphi_L(x)
=
(-1)^{d - k}
\sum_{\substack{J\subset \{1,\ldots,m\}:\\ L = \cap_{j\in J} H_j}}(-1)^{\#J}
=
(-1)^{d - k}
\sum_{\substack{\cB\subset \cA: \\ \cap_{H\in \cB} H = L}}
(-1)^{\# \cB},
\;\;\;
\text{ for all }
x\in \R^d\backslash E(L).
$$
Taking the sum over all $k$-dimensional affine subspaces $L\in \cL_k(\cA)$ generated by the arrangement $\cA$ and recalling~\eqref{eq:varphi_k_sum}, we arrive at
\begin{equation}\label{eq:varphi_k_end}
\varphi_k(x)
=
\sum_{L\in \cL_k(\cA)} \varphi_L(x)
=
(-1)^{d - k}  \sum_{\substack{\cB\subset \cA: \\ \dim (\cap_{H\in \cB} H) = k}}
(-1)^{\# \cB},
\end{equation}
for all $x\in\R^d$ such that
\begin{equation}\label{eq:x_notin}
x\notin \bigcup_{L\in \cL_k(\cA)} E(L) = \bigcup_{L\in \cL_k(\cA)} (E'(L)\cup E''(L))
\end{equation}
with
\begin{equation}\label{eq:E'L_E''L_rep}
E'(L) = \bigcup_{\substack{L_{k-1}\in \cL_{k-1}(\cA):\\L_{k-1}\subset L}} (L_{k-1} + L^\bot),
\qquad
E''(L) = \bigcup_{\substack{L_{k+1}\in \cL_{k+1}(\cA):\\L_{k+1}\supset L}} (L_{k+1}^\bot + L).
\end{equation}
By the definition of the characteristic polynomial $\chi_\cA(t)$, see~\eqref{eq:char_poly_def} and~\eqref{eq:char_poly_coeff}, the  right-hand side of~\eqref{eq:varphi_k_end} is nothing but $a_k$.  So, $\varphi_k(x) = a_k$ for all $x\in\R^d$ satisfying~\eqref{eq:x_notin}. If~\eqref{eq:x_notin} is not satisfied, we can use the inequality $\varphi_L(x) \geq a_0(L)$ to prove that $\varphi_k(x) \geq  a_k$.

\vspace*{2mm}
\noindent
\textit{Step 5.} To complete the proof of Theorem~\ref{theo:main}, it remains to check the following equality of the exceptional sets:
\begin{equation}\label{eq:exceptional_set_identity}
\bigcup_{L\in \cL_k(\cA)} (E'(L)\cup E''(L)) = \bigcup_{P\in \cR(\cA)} \bigcup_{G\in \cF_k(P)} \partial (G + N_G(P)),
\end{equation}
for all $k\in \{0,\ldots,d-1\}$. We need some preparatory lemmas.

\begin{lemma}\label{lem:dual_cones_cover}
Let $\cA= \{H_1,\ldots,H_m\}$ be a linear hyperplane arrangement in $\R^d$. Suppose that $\cA$ is of full rank meaning that $\cap_{i=1}^m H_i = \{0\}$.  Then, $\cup_{C\in \cR(\cA)} (C^\circ) = \R^d$.
\end{lemma}
\begin{proof}
By Lemma~\ref{lem:non_gen_pos_remark} it suffices to show that $a_0>0$. By Proposition~\ref{prop:k_0}, the function $\varphi_0(x) = \sum_{C\in \cR(\cA)} \ind_{C^\circ}(x)$ is Lebesgue-a.e.\ equal to $a_0$, hence $a_0\geq 0$.  Since the arrangement is of full rank, the lineality space of each chamber is trivial, that is $C\cap (-C) = \{0\}$. This implies that the dual cone $C^\circ$ has non-empty interior, hence the the function $\varphi_0(x)$  cannot be a.e.\ $0$ implying that $a_0\neq 0$.
\end{proof}

\begin{lemma}\label{lem:adjoin_vector_to_cone}
Let $C\subset \R^d$ be a polyhedral cone with a trivial lineality space, that is $C\cap (-C) = \{0\}$. Let $v\in\R^d\backslash\{0\}$ be a vector. Then, at least one of the cones $\pos (C\cup\{+v\})$ or $\pos (C\cup\{-v\})$ has a trivial lineality space.
\end{lemma}
\begin{proof}
It follows from $C\cap (-C)= \{0\}$ that there exists $\eps\in \{-1,+1\}$ such that $\eps v\notin -C$. We claim that $\pos (C\cup\{\eps v\})$ has a trivial lineality space. To prove this, take some $w$ such that both $+w$ and $-w$ are contained in $\pos (C \cup\{\eps v\})$. We then have $w= z_1 + \lambda_1 \eps v = -z_2 - \lambda_2 \eps v$ for some $z_1,z_2\in C$ and $\lambda_1,\lambda_2\geq 0$. If $\lambda_1 = \lambda_2=0$, then $z_1=-z_2$ implying that $z_1=z_2=0$ and thus  $w=0$. So, let $\lambda_1+\lambda_2>0$. Then, we have
$
\eps v = - (z_1+z_2)/ (\lambda_1 + \lambda_2) \in -C,
$
a contradiction.
\end{proof}

\begin{lemma}\label{lem:boundary_normal_cones}
Let $\cA= \{H_1,\ldots,H_m\}$ be a linear hyperplane arrangement in $\R^d$. Then,
\begin{equation}\label{eq:lem:boundary_normal_cones}
\bigcup_{L\in \cL(\cA)\backslash\{0\}} L^\bot = \bigcup_{C\in \cR(\cA)} \partial (C^\circ).
\end{equation}
\end{lemma}
\begin{proof}
If $\cA$ is not essential meaning that $L_*:=\cap_{i=1}^m H_i \neq \{0\}$, then the left-hand side of~\eqref{eq:lem:boundary_normal_cones} equals $L_*^\bot$. On the other hand, the cones $C^\circ$ are contained in $L_*^\bot$, satisfy $\partial (C^\circ) = C^\circ$, and cover the space $L_*^\bot$ by Lemma~\ref{lem:dual_cones_cover} applied to the ambient space $L_*^\bot$, thus proving that~\eqref{eq:lem:boundary_normal_cones} holds.

In the following  let $\cA$ be of full rank meaning that $\cap_{i=1}^m H_i = \{0\}$.
Let $H_1= y_1^{\bot}, \ldots, H_m=y_m^{\bot}$ for some vectors $y_1,\ldots,y_m\in \R^d\backslash\{0\}$. The linear span of $y_1,\ldots,y_m$ is $\R^d$ since the arrangement has full rank. Any subspace $L\in \cL(\cA)$ has the form $L = \cap_{i\in I} H_i = \lin \{y_i:i\in I\}^\bot$ for some set $I\subset \{1,\ldots,m\}$. The corresponding orthogonal complement is $L^\bot =  \lin\{y_i:i\in I\}$. It follows that
$$
\bigcup_{L\in \cL(\cA)\backslash\{0\}} L^\bot = \bigcup_{\substack{I\subset \{1,\ldots,m\}:\\ \lin \{y_i: i\in I\}\neq \R^d}} \lin \{y_i: i\in I\}.
$$
To complete the proof, we need to show that
\begin{equation}\label{eq:cup_equals_cup}
\bigcup_{\substack{I\subset \{1,\ldots,m\}:\\ \lin \{y_i: i\in I\}\neq \R^d}} \lin \{y_i: i\in I\} = \bigcup_{C\in \cR(\cA)} \partial (C^\circ).
\end{equation}

To prove the inclusion $\subset$, let $v\in \lin\{y_i:i\in I\}\neq \R^d$ for some $I\subset \{1,\ldots,m\}$. By first extending  $I$ and then excluding the superfluous linearly dependent elements, we may assume that $M:=\lin\{y_i: i\in I\}$ has dimension $d-1$ and that the vectors $\{y_i:i\in I\}$ are linearly independent. We can find $\eps_i\in \{-1,+1\}$, for all $i\in I$, such that $v\in \pos\{\eps_i y_i:i\in I\}$.  Let $M_+$ and $M_-$ be the closed half-spaces in which the hyperplane $M$ dissects $\R^d$. Let $J_1$, respectively $J_2$, be the set of all $j\in \{1,\ldots,m\}\backslash I$ such that $y_j\in M$, respectively $y_j\in \R^d\backslash M$. The cone $\pos\{\eps_i y_i:i\in I\}\subset M$ has a trivial lineality space because $\{\eps_i y_i:i\in I\}$ is a basis of $M$.  By inductively applying Lemma~\ref{lem:adjoin_vector_to_cone} in the ambient space $M$, we can find $\eps_j\in \{-1,+1\}$, for all $j\in J_1$, such that the cone $D:= \pos\{\eps_i y_i: i\in I\cup J_1\}\subset M$ has a trivial lineality space. Furthermore, for every $j\in J_2$ we can find $\eps_j\in \{-1,+1\}$ such that $\eps_j y_j \in \Int M_+$.
With the signs $\eps_1,\ldots,\eps_m\in \{-1,+1\}$ constructed as above, we consider the cone
\begin{equation}\label{eq:proof_exc_sets_C}
C:= \{z\in \R^d: \langle z, \eps_1 y_1\rangle \leq 0, \ldots, \langle z, \eps_m y_m\rangle \leq 0\}.
\end{equation}
The dual cone is the positive hull
\begin{equation}\label{eq:proof_exc_sets_C_circ}
C^\circ = \pos \{\eps_1 y_1,\ldots, \eps_m y_m\}.
\end{equation}
By construction, $C^\circ \subset M_+$ and $C^\circ\cap M = D$. Also, the cone $C^\circ$ has a trivial lineality space because $\pm w\in C^\circ$ would imply $\pm w \in C^\circ \cap M = D$, which implies $w=0$ because $D$ has a trivial lineality space by construction. By duality, $C$ has non-empty interior. It follows that $C$ is a chamber of the arrangement $\cR(\cA)$. By construction, $C^\circ\subset M_+$ and $v\in C^\circ\cap M$, hence $v\in \partial (C^\circ)$, thus completing the proof of the inclusion $\subset$ in~\eqref{eq:cup_equals_cup}.

To prove the inclusion $\supset$ in~\eqref{eq:cup_equals_cup}, take any  $C\in \cR(\cA)$ and any $v\in \partial (C^\circ)$. Then, $C$ and $C^\circ$ must be of the same form as in~\eqref{eq:proof_exc_sets_C} and~\eqref{eq:proof_exc_sets_C_circ}. Moreover, since $C$ has non-empty interior, the lineality space of the cone $C^\circ$ is trivial. If $v\in \partial (C^\circ)$, then $v\in F$ for some face $F\in \cF(C^\circ)$  of dimension $d-1$. Let $I$ be the set of all $i\in \{1,\ldots,m\}$ with $\eps_i y_i\in F$. Then, we have  $\lin \{\eps_i y_i: i\in I\} = \lin F$, which contains $v$ and does not coincide with $\R^d$. It follows that $v$ belongs to the left-hand side of~\eqref{eq:cup_equals_cup}, thus completing the proof.
\end{proof}

Now we are in position to prove~\eqref{eq:exceptional_set_identity}.  We have
\begin{align*}
\bigcup_{P\in \cR(\cA)} \bigcup_{G\in \cF_k(P)} \partial (G + N_G(P))
&=
\bigcup_{L\in \cL_k(\cA)} \bigcup_{\substack{G\in \cR_k(\cA):\\G\subset L}} \bigcup_{\substack{P\in \cR(\cA):\\G\in \cF_k(P)}}
 \left((\partial G + N_G(P))\cup (G + \partial N_G(P))\right)\\
&=
\bigcup_{L\in \cL_k(\cA)} (H'(L) \cup H''(L))
\end{align*}
with
$$
H'(L) = \bigcup_{\substack{G\in \cR_k(\cA):\\G\subset L}}  \Big( \partial G + \bigcup_{\substack{P\in \cR(\cA):\\G\in \cF_k(P)}}  N_G(P) \Big),
\;\;\;
H''(L) = \bigcup_{\substack{G\in \cR_k(\cA):\\G\subset L}}  \Big(G + \bigcup_{\substack{P\in \cR(\cA):\\G\in \cF_k(P)}}  \partial N_G(P) \Big).
$$

We claim that $H'(L) = E'(L)$. To prove this it suffices to show that
for every $G\in \cR_k(\cA)$ such that $G\subset L$ we have $\cup_{P\in \cR(\cA):G\in \cF_k(P)}  N_G(P)= L^\bot$. In~\eqref{eq:N_G_equals_T_circ} we characterized the normal cones $N_G(P)$ as the dual cones of the chambers of some  essential (full rank) linear hyperplane arrangement $\cA(L)$ in $L^\bot$. These dual cones cover $L^\bot$ by Lemma~\ref{lem:dual_cones_cover}, thus proving the claim.

It remains to show that $H''(L) = E''(L)$.  To this end, it suffices to prove that for every $G\in \cR_k(\cA)$ such that $G\subset L$ we have
\begin{equation}\label{eq:proof_exc_set_claim}
\bigcup_{\substack{P\in \cR(\cA):\\G\in \cF_k(P)}}  \partial N_G(P) = \bigcup_{\substack{L_{k+1}\in \cL_{k+1}(\cA):\\  L_{k+1}\supset L}} L_{k+1}^\bot.
\end{equation}
Again, recall from~\eqref{eq:N_G_equals_T_circ} that the normal cones $N_G(P)$ are the dual cones of the chambers of the linear full-rank hyperplane arrangement $\cA(L)= \{y_1^\bot \cap L^\bot,\ldots, y_\ell^{\bot}\cap L^\bot\}$ in $L^\bot$. Applying Lemma~\ref{lem:boundary_normal_cones} to this arrangement, we obtain
$$
\bigcup_{\substack{P\in \cR(\cA):\\G\in \cF_k(P)}}  \partial N_G(P) = \bigcup_{M\in \cL(\cA(L))\backslash\{0\}}(M^\bot\cap L^\bot).
$$
The right-hand side coincides with the set $E_0^*$ defined in~\eqref{eq:E_0*_def2}.  Thus, the claim~\eqref{eq:proof_exc_set_claim} follows from the identity already established in~\eqref{eq:E_0*}.
\hfill $\Box$

\section*{Acknowledgements}
Supported by the German Research Foundation under Germany's Excellence Strategy  EXC 2044 -- 390685587, Mathematics M\"unster: Dynamics - Geometry - Structure.


\bibliography{coeff_bib}

\begin{thebibliography}{24}
\providecommand{\natexlab}[1]{#1}
\providecommand{\url}[1]{\texttt{#1}}
\expandafter\ifx\csname urlstyle\endcsname\relax
  \providecommand{\doi}[1]{doi: #1}\else
  \providecommand{\doi}{doi: \begingroup \urlstyle{rm}\Url}\fi

\bibitem[Amelunxen and Lotz(2017)]{amelunxen_lotz}
D.~Amelunxen and M.~Lotz.
\newblock Intrinsic volumes of polyhedral cones: a combinatorial perspective.
\newblock \emph{Discrete Comput. Geom.}, 58\penalty0 (2):\penalty0 371--409,
  2017.
\newblock \doi{10.1007/s00454-017-9904-9}.
\newblock URL \url{https://doi.org/10.1007/s00454-017-9904-9}.

\bibitem[Amelunxen et~al.(2014)Amelunxen, Lotz, McCoy, and Tropp]{ALMT14}
D.~Amelunxen, M.~Lotz, M.~McCoy, and J.~Tropp.
\newblock Living on the edge: Phase transitions in convex programs with random
  data.
\newblock \emph{Inform. Inference}, 3:\penalty0 224--294, 2014.

\bibitem[B\'{o}na(2015)]{bona_handbook}
Mikl\'{o}s B\'{o}na, editor.
\newblock \emph{Handbook of enumerative combinatorics}.
\newblock Discrete Mathematics and its Applications (Boca Raton). CRC Press,
  Boca Raton, FL, 2015.

\bibitem[Cowan(2007)]{cowan1}
R.~Cowan.
\newblock Identities linking volumes of convex hulls.
\newblock \emph{Adv. in Appl. Probab.}, 39\penalty0 (3):\penalty0 630--644,
  2007.
\newblock \doi{10.1239/aap/1189518631}.
\newblock URL \url{https://doi.org/10.1239/aap/1189518631}.

\bibitem[De~Concini and Procesi(2006)]{deconcini}
C.~De~Concini and C.~Procesi.
\newblock A curious identity and the volume of the root spherical simplex.
\newblock \emph{Atti Accad. Naz. Lincei Rend. Lincei Mat. Appl.}, 17\penalty0
  (2):\penalty0 155--165, 2006.
\newblock \doi{10.4171/RLM/460}.
\newblock URL \url{https://doi.org/10.4171/RLM/460}.
\newblock With an appendix by John R. Stembridge.

\bibitem[Denham(2008)]{denham}
G.~Denham.
\newblock A note on {D}e {C}oncini and {P}rocesi's curious identity.
\newblock \emph{Atti Accad. Naz. Lincei Rend. Lincei Mat. Appl.}, 19\penalty0
  (1):\penalty0 59--63, 2008.
\newblock \doi{10.4171/RLM/507}.
\newblock URL \url{https://doi.org/10.4171/RLM/507}.

\bibitem[{Drton} and {Klivans}(2010)]{drton_klivans}
M.~{Drton} and C.~J. {Klivans}.
\newblock {A geometric interpretation of the characteristic polynomial of
  reflection arrangements.}
\newblock \emph{{Proc. Am. Math. Soc.}}, 138\penalty0 (8):\penalty0 2873--2887,
  2010.
\newblock \doi{10.1090/S0002-9939-10-10369-4}.

\bibitem[Glasauer(1995)]{glasauer_phd}
S.~Glasauer.
\newblock {Integralgeometrie konvexer K\"orper im sph\"arischen Raum}.
\newblock PhD Thesis, University of Freiburg. Available at:
  \url{http://www.hs-augsburg.de/~glasauer/publ/diss.pdf}, 1995.

\bibitem[Glasauer(1998)]{glasauer_local_steiner}
S.~Glasauer.
\newblock An {E}uler-type version of the local {S}teiner formula for convex
  bodies.
\newblock \emph{Bull. London Math. Soc.}, 30\penalty0 (6):\penalty0 618--622,
  1998.
\newblock \doi{10.1112/S0024609398004858}.
\newblock URL \url{https://doi.org/10.1112/S0024609398004858}.

\bibitem[Hug(1999)]{hug_habil}
D.~Hug.
\newblock {Measures, curvatures and currents in convex geometry}.
\newblock Habilitation thesis, {U}niversity of Freiburg, 1999.

\bibitem[Hug and Kabluchko(2018)]{hug_kabluchko}
D.~Hug and Z.~Kabluchko.
\newblock An inclusion-exclusion identity for normal cones of polyhedral sets.
\newblock \emph{Mathematika}, 64\penalty0 (1):\penalty0 124--136, 2018.
\newblock \doi{10.1112/S0025579317000390}.
\newblock URL \url{https://doi.org/10.1112/S0025579317000390}.

\bibitem[Kabluchko et~al.(2017{\natexlab{a}})Kabluchko, Last, and
  Zaporozhets]{kabluchko_last_zaporozhets}
Z.~Kabluchko, G.~Last, and D.~Zaporozhets.
\newblock Inclusion-exclusion principles for convex hulls and the {E}uler
  relation.
\newblock \emph{Discrete Comput. Geom.}, 58\penalty0 (2):\penalty0 417--434,
  2017{\natexlab{a}}.
\newblock \doi{10.1007/s00454-017-9880-0}.
\newblock URL \url{https://doi.org/10.1007/s00454-017-9880-0}.

\bibitem[Kabluchko et~al.(2017{\natexlab{b}})Kabluchko, Vysotsky, and
  Zaporozhets]{KVZ15}
Z.~Kabluchko, V.~Vysotsky, and D.~Zaporozhets.
\newblock Convex hulls of random walks, hyperplane arrangements, and {W}eyl
  chambers.
\newblock \emph{Geom. Funct. Anal.}, 27\penalty0 (4):\penalty0 880--918,
  2017{\natexlab{b}}.

\bibitem[{Klivans} and {Swartz}(2011)]{klivans_swartz}
C.~J. {Klivans} and E.~{Swartz}.
\newblock {Projection volumes of hyperplane arrangements.}
\newblock \emph{{Discrete Comput. Geom.}}, 46\penalty0 (3):\penalty0 417--426,
  2011.
\newblock \doi{10.1007/s00454-011-9363-7}.

\bibitem[Lofano and Paolini(2018)]{lofano_paolini}
D.~Lofano and G.~Paolini.
\newblock Euclidean matchings and minimality of hyperplane arrangements.
\newblock Preprint at arXiv: 1809.02476, 2018.

\bibitem[McMullen(1975)]{mcmullen}
P.~McMullen.
\newblock Non-linear angle-sum relations for polyhedral cones and polytopes.
\newblock \emph{Math. Proc. Cambridge Philos. Soc.}, 78\penalty0 (2):\penalty0
  247--261, 1975.
\newblock \doi{10.1017/S0305004100051665}.
\newblock URL \url{https://doi.org/10.1017/S0305004100051665}.

\bibitem[Miles(1959)]{miles_amalgamation}
R.~E. Miles.
\newblock The complete amalgamation into blocks, by weighted means, of a finite
  set of real numbers.
\newblock \emph{Biometrika}, 46:\penalty0 317--327, 1959.
\newblock \doi{10.1093/biomet/46.3-4.317}.
\newblock URL \url{https://doi.org/10.1093/biomet/46.3-4.317}.

\bibitem[Padberg(1999)]{padberg_book}
M.~Padberg.
\newblock \emph{Linear optimization and extensions}, volume~12 of
  \emph{Algorithms and Combinatorics}.
\newblock Springer-Verlag, Berlin, expanded edition, 1999.
\newblock \doi{10.1007/978-3-662-12273-0}.
\newblock URL \url{https://doi.org/10.1007/978-3-662-12273-0}.

\bibitem[Rockafellar(1970)]{rockafellar_book}
R.~T. Rockafellar.
\newblock \emph{Convex analysis}.
\newblock Princeton Mathematical Series, No. 28. Princeton University Press,
  Princeton, N.J., 1970.

\bibitem[Schneider(2014)]{SchneiderBook}
R.~Schneider.
\newblock \emph{Convex {B}odies: the {B}runn-{M}inkowski {T}heory}, volume 151
  of \emph{Encyclopedia of Mathematics and its Applications}.
\newblock Cambridge University Press, Cambridge, expanded edition, 2014.

\bibitem[Schneider(2018)]{schneider_combinatorial}
R.~Schneider.
\newblock Combinatorial identities for polyhedral cones.
\newblock \emph{St. Petersburg Math. J.}, 29\penalty0 (1):\penalty0 209--221,
  2018.
\newblock \doi{10.1090/spmj/1489}.
\newblock URL \url{https://doi.org/10.1090/spmj/1489}.

\bibitem[Schneider and Weil(2008)]{SW08}
R.~Schneider and W.~Weil.
\newblock \emph{Stochastic and {I}ntegral {G}eometry}.
\newblock Probability and its Applications. Springer--Verlag, Berlin, 2008.

\bibitem[Stanley(2007)]{stanley_book}
R.~P. Stanley.
\newblock An introduction to hyperplane arrangements.
\newblock In \emph{Geometric combinatorics}, volume~13 of \emph{IAS/Park City
  Math. Ser.}, pages 389--496. Amer. Math. Soc., Providence, RI, 2007.

\bibitem[Ziegler(1995)]{ziegler_book_lec_on_poly}
G.~M. Ziegler.
\newblock \emph{Lectures on polytopes}, volume 152 of \emph{Graduate Texts in
  Mathematics}.
\newblock Springer-Verlag, New York, 1995.
\newblock \doi{10.1007/978-1-4613-8431-1}.

\end{thebibliography}
\bibliographystyle{plainnat}
\vspace{1cm}

\end{document}